\documentclass[11pt]{amsart}
\usepackage[top=30truemm,bottom=30truemm,left=30truemm,right=30truemm]{geometry} 
\usepackage{amsmath, amssymb, array}
\usepackage[all]{xy}
\usepackage{ascmac}
\usepackage{color}
\usepackage{amscd}
\usepackage{fancyhdr}
\usepackage[alphabetic]{amsrefs}

\title{A unique pair of triangles} 
\thanks{Published online(https://doi.org/10.1016/j.jnt.2018.07.007)}
\thanks{0022-314X/$\copyright$ 2018 The Authors. Published by Elsevier Inc. This is an open access article under the CC BY-NC-ND license (http://creativecommons.org/licenses/by-nc-nd/4.0/).}

\author{Yoshinosuke Hirakawa \and Hideki Matsumura}
\address[Hideki Matsumura]{Department of Science and Technology, Keio University, 14-1, Hiyoshi 3-chome, Kouhoku-ku, Yokohama-shi, Kanagawa-ken, Japan}
\email{hirakawa@keio.jp}
\email{hidekimatsumura@keio.jp}

\thanks{This research was supported by JSPS KAKENHI Grant Number JP15J05818 and the Research Grant of Keio Leading-edge Laboratory of Science $\&$ Technology (Grant Numbers 000036 and 000053). This research was also conducted as part of the KiPAS program FY2014-2018 of the Faculty of Science and Technology at Keio University, and was supported in part by JSPS KAKENHI 26247004, 18H05233, as well as the JSPS Core-to-Core program ``Foundation of a Global Research Cooperative Center in Mathematics focused on Number Theory and Geometry".}
\date{\today}
 
\makeatletter
  \def\section{\@startsection{section}{1}{\z@}%
     {4ex plus 0ex}%
     {4ex plus 0ex}%
     {\normalfont\large\bfseries}}
  \makeatother

\makeatletter
  \def\subsection{\@startsection{subsection}{1}{\z@}%
     {2ex plus 0ex}%
     {2ex plus 0ex}%
     {\normalfont\normalsize\bfseries}}
  \makeatother

\theoremstyle{plain}
 \newtheorem{theorem}{Theorem}[section]

\theoremstyle{definition}

\begin{document}

%%%Title&Author

\maketitle

%%%%%
%%%%% Abstract 
%%%%%
\begin{abstract}
A rational triangle is a triangle with sides of rational lengths. In this short note, we prove that there exists a unique pair of a rational right triangle and a rational isosceles triangle which have the same perimeter and the same area. In the proof, we determine the set of rational points on a certain hyperelliptic curve by a standard but sophisticated argument which is based on the 2-descent on its Jacobian variety and Coleman's theory of $p$-adic abelian integrals.
\end{abstract}

%%%%%
%%%%% Body
%%%%%
\section{Main theorem and its proof}

A rational (resp. integral) triangle is a triangle with sides of rational (resp. integral) lengths. Such a triangle has arithmetic interest: For instance, every rational right triangle has sides of lengths $(k(1+x^{2}), k(1-x^{2}), 2kx)$ with positive rational numbers $k, x > 0$. We may check this fact from the Pythagorean theorem and the uniqueness of the prime decomposition, which are most elementary theorems in geometry and arithmetic respectively.

From the point of view of arithmetic, perimeter and area are fundamental invariants of a rational triangle. Therefore, it is natural to try to classify rational triangles by their perimeters and/or areas. Indeed, there are several works on construction of infinitely many pairs of rational triangles which have the same perimeters and the same areas (see e.g. \cite{Bremner}, \cite{van Luijk} and references there).

A primitive triangle is an integral triangle such that the greatest common divisor of the lengths of its sides is 1. We can prove that there exists no pair of a primitive right triangle and a primitive isosceles triangle which have the same perimeter and the same area. We give an elementary proof of this fact in Appendix. How many such pairs are there in the category of rational triangles? In this short note, we give the complete answer to this question, that is, there exists only one such pair of triangles.

\begin{theorem} \label{main theorem}
Up to similitude, there exists a unique pair of a rational right triangle and a rational isosceles triangle which have the same perimeter and the same area. The unique pair consists of the right triangle with sides of lengths $(377, 135, 352)$ and the isosceles triangle with sides of lengths $(366, 366, 132)$.
\end{theorem}

Since the proof for the primitive case is quite elementary, it may be surprising that our proof of Theorem \ref{main theorem} given below depends on a sophisticated theory of modern arithmetic geometry of hyperelliptic curves (of genus 2), which contains the 2-descent argument on the Jacobian variety of a hyperelliptic curve (\cite{FPS}, \cite{Stoll}) and Coleman's theory of $p$-adic abelian integrals (\cite{Coleman Annals}, \cite{Coleman Duke}). In particular, the following theorem of Chabauty and Coleman is a key ingredient of our proof:

\begin{theorem} [{\cite{Chabauty}, \cite[Corollary 4a]{Coleman Duke}, \cite[Theorem 5.3 (b)]{MP}}] \label{Chabauty-Coleman}
Let $C$ be a curve
%footnote
\footnote{Here, a curve over $\mathbb{Q}$ means a smooth projective geometrically integral scheme of dimension 1 over $\mathrm{Spec}(\mathbb{Q})$.}
%end
of genus $g \geq 2$ over $\mathbb{Q}$, and $J$ be its Jacobian variety. Let $p$ be a prime number. Suppose that the rank of $J(\mathbb{Q})$ is smaller than $g$, and $p > 2g$ and $C$ has good reduction at $p$. Then, we have
\begin{eqnarray*}
\#C(\mathbb{Q}) \leq \#C(\mathbb{F}_{p})+(2g-2).
\end{eqnarray*}
\end{theorem}

Although there are several works on rational triangles, the only previous work which uses such a theory is \cite{ZP}, where the authors  prove that there exists no pair of an integral isosceles triangle and a certain integral rhombus which have the same perimeters and the same integral areas.

\begin{proof} [Proof of Theorem \ref{main theorem}]
Assume that there exists such a pair of triangles. First, note that since every rational isosceles triangle with a rational area has a rational height, it is a union of the two copies of a rational right triangle.  Moreover, by rescaling both of the given triangles, we may assume that the given triangles have sides of lengths
\begin{enumerate}
\item $(k(1+x^{2}), k(1-x^{2}), 2kx)$ and $(1+u^{2}, 1+u^{2}, 4u)$, or
\item $(k(1+x^{2}), k(1-x^{2}), 2kx)$ and $(1+u^{2}, 1+u^{2}, 2(1-u^2))$
\end{enumerate}
respectively with positive rational numbers $x, u, k$.
%footnote
\footnote{Cf. the proof of Theorem \ref{appendix}.}
%end

In case (1), we have a simultaneous equation
\begin{eqnarray*}
\begin{cases}
k+kx = 1+2u+u^{2} \\
k^{2}x(1-x^{2}) = 2u(1-u^{2}).
\end{cases}
\end{eqnarray*}
Since $x, u, k > 0$, this is equivalent to
\begin{eqnarray*}
\begin{cases}
k(1+x) = w^{2} \\
(w^{2}-k)w(2k-w^{2})= 2k(w-1)(w-2),
\end{cases}
\end{eqnarray*}
where we set $w = u+1$. Since the former equation has a unique solution $(x, w, k)$ for every solution $(w, k)$ with $w > 1, k > 0$ of the latter, this simultaneous equation is equivalent to the single latter equation under the condition $w > 1, k > 0$. Moreover, it is equivalent to
\begin{eqnarray*}
-w^{5}+3kw^{3}-2k^{2}w = 2kw^{2}-6kw+4k,
\end{eqnarray*}
i.e.,
\begin{eqnarray*}
2wk^{2}+(-3w^{3}+2w^{2}-6w+4)k+w^{5} = 0.
\end{eqnarray*}
Since $k$ is a rational number, the discriminant of the left hand side as a polynomial of $k$ is a square integer, say, $r^{2}$. Therefore, we obtain
\begin{eqnarray*}
r^{2} = (-3w^{3}+2w^{2}-6w+4)^{2}-8w^{6},
\end{eqnarray*}
which defines an affine curve. We denote its non-singular compactification by $C_{1}$, which is a hyperelliptic curve.

It is easy to check that $C_{1}$ has at least ten rational points, that is, $(w, r) = (0, \pm4), (1, \pm1)$, $(2, \pm8), (12, \pm868)$ and two points at infinity. All of these points do not give triangles.
%footnote
\footnote{Indeed, the equalities $(w, r) = (0, \pm4), (1, \pm1)$ imply that the bottom length $4u$ of the isosceles triangle is less than or equal to zero, and $(w, r) = (2, \pm8), (12, \pm868)$ imply that the height $1-u^{2}$ of the isosceles triangle is less than or equal to zero.}
%end

On the other hand, we can prove that the Mordell-Weil rank of the Jacobian variety of $C_{1}$ over $\mathbb{Q}$ is at most 1.
%footnote
\footnote{In fact, it has exactly rank 1.}
%end
Indeed, we can check it by Magma Calculator
%footnote
\footnote{http://magma.maths.usyd.edu.au/calc}
%end
(cf. \cite{Bosma-Cannon-Playoust}). Here are the inputs and output, where the algorithm is based on \cite{Stoll}.
%footnote
\footnote{See also the Magma handbook: Example CrvHyp\_{}sha\_{}visibility (H131E34) in https://magma.maths.usyd.edu.au/magma/handbook/text/1501.}
%end

$>$ R$<$w$>$:=PolynomialRing(Rationals());

$>$ C:=HyperellipticCurve((-3*w\^{}3+2*w\^{}2-6*w+4)\^{}2-8*w\^{}6);

$>$ J:=Jacobian(C);

$>$ RankBound(J);

$1$

Finally, we may easily check that $C_{1}$ has good reduction at 5
%\footnote{In fact, the discriminant of the sextic polynomial $(s^{2}-2)(s^{4}-6s^{3}+15s^{2}-12s-2)$ of $s$ is $-2^{27} \cdot 47$, hence the discriminant of (the above model of) $C$ is $-2^{15} \cdot 47$.}
and $\#C_{1}(\mathbb{F}_{5}) = 8$.
Therefore, by applying Theorem \ref{Chabauty-Coleman}, we obtain that $\#C_{1}(\mathbb{Q}) \leq 10$, i.e., there exists no other rational point than the above ten rational points.

Next, we consider case (2). We have a simultaneous equation
\begin{eqnarray*}
\begin{cases}
k+kx = 2 \\
k^2x(1-x^{2}) = 2u(1-u^{2}).
\end{cases}
\end{eqnarray*}
Since $x, u, k > 0$, this is equivalent to
\begin{eqnarray*}
\begin{cases}
k(1+x) = 2 \\
(2-k)(2k-2) = ku(1-u^{2}).
\end{cases}
\end{eqnarray*}
Since the former equation has a unique solution $(x, u, k)$ for every solution $(u, k)$ with $0<k<2$ of the latter, this simultaneous equation is equivalent to the single latter equation under the condition $0<k<2$. Moreover, it is equivalent to
\begin{eqnarray*}
2k^{2}-(u^{3}-u+6)k+4 = 0.
\end{eqnarray*}
Since $k$ is a rational number, the discriminant of the left hand side as a polynomial of $k$ is a square integer, say, $s^{2}$. Therefore, we obtain
\begin{eqnarray*}
s^{2} = (u^{3}-u+6)^{2}-32,
\end{eqnarray*}
which defines an affine curve. We denote its non-singular compactification by $C_{2}$, which is a hyperelliptic curve.
%footnote
\footnote{In fact, the two curves $C_{1}$ and $C_{2}$ are isomorphic to each other via the isomorphism induced by a birational map given by $(u, s) = (1-2/w, 2r/w^{3})$.}
%end
It is easy to check that $C_{2}$ has at least ten rational points, that is, $(u, s) = (0, \pm2), (1, \pm2)$, $(-1, \pm2), (5/6, \pm217/216)$ and two points at infinity. The former six points do not give triangles.
%footnote
\footnote{Indeed, the equalities $(u, s) = (0, \pm2)$ imply that the height $2u$ of the isosceles triangle is zero, and $(u, s) = (1, \pm2), (-1, \pm2)$ imply that the bottom length $2(1-u^{2})$ of the isosceles triangle is zero.}
%end
Moreover, the equalities $(u, s) = (5/6, 217/216)$ and $(5/6, -217/216)$ imply that $(k, x, u) = (27/16, 5/27, 5/6)$ and $(32/27, 11/16, 5/6)$ respectively, both of which give the unique pair in the statement up to similitude. Finally, by the same argument as in case (1), we may also check that $\#C_{2}(\mathbb{Q}) \leq 10$, i.e., there exists no other rational point than the above ten rational points. This completes the proof.
\end{proof}

\section*{Acknowledgements}

The authors thank their advisor Ken-ichi Bannai for reading the draft and giving many helpful comments. The authors also thank him for warm and constant encouragement. The authors also thank the referee for her or his valuable comments which made this article clearer.

\section*{Appendix : No pair of a primitive right triangle and an isosceles triangle}

\begin{theorem} \label{appendix}
There exists no pair of a primitive right triangle and a primitive isosceles triangle which have the same perimeter and the same area.
\end{theorem}

\begin{proof}
We prove this statement by contradiction. Assume that there exists such a pair of triangles. As we have done in the proof of Theorem \ref{main theorem}, we may assume that the given triangles have sides of lengths
\begin{enumerate}
\item $(x^{2}+y^{2}, x^{2}-y^{2}, 2xy)$ and $(u^{2}+v^{2}, u^{2}-v^{2}, 4uv)$, or
\item $(x^{2}+y^{2}, x^{2}-y^{2}, 2xy)$ and $(u^{2}+v^{2}, u^{2}-v^{2}, 2(u^2-v^{2}))$
\end{enumerate}
respectively with positive integers $x, y, u, v$ such that $x$ and $y$ (resp. $u$ and $v$) are coprime, and exactly one of $x$ and $y$ (resp. $u$ and $v$) are even.

In the case (1), we have a simultaneous equation
\begin{eqnarray*}
\begin{cases}
2x^{2}+2xy = 2u^{2}+4uv+2v^{2} \\
xy(x^{2}-y^{2}) = 2uv(u^{2}-v^{2}).
\end{cases}
\end{eqnarray*}
Since $x, x+y, u+v > 0$, this is equivalent to
\begin{eqnarray*}
\begin{cases}
x(x+y) = (u+v)^{2} \\
y(x-y) = \frac{2uv(u-v)}{u+v}.
\end{cases}
\end{eqnarray*}
Here, note that since $y(x-y)$ is a positive integer, $2uv(u-v)/(u+v)$ is also a positive integer. On the other hand, since either $u$ or $v$ is even, we see that $u+v$ is an odd positive integer. Moreover, since $u$ and $v$ are coprime, $u+v$ and $uv$ are also coprime. Since $2uv(u-v)/(u+v)$ is a positive integer, this implies that $(u-v)/(u+v)$ is a positive integer, which is impossible whenever $u, v > 0$.

The proof for case (2) is the same. This completes the proof.
\end{proof}

\end{document}